\documentclass[a4paper,12pt]{article}
\usepackage{times, url}
\usepackage[utf8]{inputenc}
\usepackage[T1]{fontenc}
\usepackage{etoolbox}
\usepackage{amsmath,amsthm,amsfonts,latexsym,amssymb,mathrsfs,graphics,pst-all,bbm,amscd,xypic,verbatim,epsfig,pstricks,makeidx}
\usepackage{amsthm, enumitem}
\usepackage{mathtools}

\newtheorem{theorem}{Theorem}[section]
\newtheorem{lemma}[theorem]{Lemma}

\newtheorem{definition}{Definition}[section]
\newtheorem{corollary}[theorem]{Corollary}
\newtheorem{example}{Example}[section]
\newtheorem{remark}[theorem]{Remark}

\usepackage[center]{caption}
\usepackage[none]{hyphenat}
\usepackage{datetime}

\newdateformat{monthdayyeardate}{%
  \monthname[\THEMONTH]~\THEDAY, \THEYEAR}

\usepackage{color, colortbl}
\definecolor{Gray}{gray}{0.9}
\usepackage{amsthm, enumitem}

\makeatother

\begin{document}
\setcounter{page}{1}

\begin{center}
{\LARGE \bf  On Sombor coindex of graphs}
\vspace{8mm}

{\bf Phanjoubam Chinglensana$^1$, Sainkupar Mn Mawiong$^2$ }
\vspace{3mm}

$^1$Department of Mathematics, North-Eastern Hill University,\\  
NEHU Campus, Shillong-793022, INDIA\\
e-mail: \url{phanjoubam17@gmail.com}
\vspace{2mm}

$^2$ Department of Basic Sciences and Social Sciences, North-Eastern Hill University,\\
  NEHU Campus, Shillong-793022, INDIA\\
e-mail: \url{skupar@gmail.com}
\vspace{2mm}

\monthdayyeardate\today

\end{center}
\vspace{10mm}

{\bf Abstract:}
Motivated by the recently introduced topological index, the Somber index, we define a new topological index of a graph in this paper, we call it Sombor coindex. The Sombor coindex is defined by considering analogous contributions from the pairs of non-adjacent vertices, capturing, thus, and quantifying a possible influence of remote pairs of vertices. We give several properties of the Somber coindex and its relations to the Sombor index, the Zagreb (co)indices, forgotten coindex and other important graph parameters. We also compute the bounds of the Somber coindex of some graph operations and compute the Sombor coindex for some chemical graphs as an application.
\vspace{5mm}

\section{Introduction}
Chemical graph theory generally considers various graph-theoretical invariants of molecular graphs (also known as topological indices or molecular descriptors), and study how strongly are they correlated with various properties of the corresponding molecules. Along this line of approach, numerous graph invariants have been employed with varying degree of success in QSAR (quantitative structure-activity relationship) and QSPR (quantitative structure-property relationship) studies. The Zagreb indices are among the more useful invariants and they are defined as sums of contributions dependent on the degrees of adjacent vertices over all edges of a graph. The Zagreb indices of a graph $G$, i.e., the first Zagreb index $M_1(G)$ and the second Zagreb index $M_2(G)$, were originally defined as follows.\\
$$M_1(G)=\sum_{u\in V(G)} {d_G(u)}^2; \ \ M_2(G)=\sum_{uv\in E(G)} d_G(u)d_G(v).$$
The first Zagreb index of $G$ can also be expressed as $$M_1(G)=\sum_{uv\in E(G)} [d_G(u)+d_G(v)].$$
The Zagreb indices can be viewed as the contributions of pairs of adjacent vertices to additively and multiplicatively weighted versions of Wiener numbers and polynomials {\cite9}. When computing the weighted Wiener polynomials of certain composite graphs, similar contributions of non-adjacent pairs of vertices had to be taken into account {\cite6}. These quantities were called Zagreb coindices since the defining sums run over the edges of the complement of $G$ although the degrees of the vertices are with respect to the graph itself. The first Zagreb coindex $\overline{M}_1(G)$ and the second Zagreb coindex $ \overline{M}_2(G) $ are defined formally in {\cite6} as follows.
$$\overline{M}_1(G)=\sum_{uv\notin E(G)} [d_G(u)+d_G(v)]; \ \ \overline{M}_2(G)=\sum_{uv\notin E(G)} d_G(u)d_G(v).$$
Generalised version of the first Zagreb index have also been introduced and is known as the general first Zagreb index which is defined as 
$$M_1^p (G)=\sum_{u\in V(G)} {d_G(u)}^p.$$ When $p=3$, $M_1^3 (G) =\displaystyle\sum_{u\in V(G)} {d_G(u)}^3$ is known as the forgotten index, denoted by $F(G)$, and is also equal to 
$$F(G)=\sum_{uv\in E(G)}[ {d_G(u)}^2+ {d_G(v)}^2].$$ The forgotten coindex of a graph $G$ is defined \cite5 as $$\overline{F}(G)=\sum_{uv\notin E(G)} [{d_G(u)}^2+{d_G(v)}^2].$$

Several properties of the Zagreb coindices have been explored and the Zagreb coindices for some derived graphs and several graph operations have also been studied recently ({\cite1}, {\cite{11}}). In \cite5, forgotten coindex of some graph operations are obtained and in \cite2 relations between this invariant and some well-known graph invariants are explored. Motivated by such results, we consider in this paper the recently introduced topological index of a graph, the Sombor index. We define, analogous to Zagreb coindices, the Sombor coindex of a graph. We give several properties of the Sombor coindex and its relations to the Sombor index, Zagreb (co)indices, forgotten coindex and other important graph parameters. Since several complicated (and important) graphs often arise from simpler graphs via some graph operations, we also present the Sombor coindex of some graph operations.
The paper is arranged as follows: In section 2, we recall the definition of the Sombor index and we define the Sombor coindex along with some examples. In section 3, we give several properties of the Sombor coindex and its bounds in terms of important graph parameters. In section 4, we explore the relations between Sombor coindex and other topological coindices: Zagreb coindices and forgotten coindex. In section 5, we present the Sombor coindex of some graph operations and we compute the Sombor coindex of some (chemical) graphs as application.

We consider only finite simple graph in this paper. Let $G$ be a finite simple graph on $n$ vertices and $m$ edges. We denote the vertex set and the edge set of $G$ by $V(G)$ and $E(G)$, respectively. The complement of $G$, denoted by $\overline{G}$, is a simple graph on $V(G)$ in which two vertices $u$ and $v$ are adjacent, i.e., joined by an edge $uv$, if and only if they are not adjacent in $G$. Hence, $uv\in E(\overline{G})$ if and only if $uv\notin E(G)$. Clearly, $E (G)\cup E(\overline{G}) = E(K_n )$, where $K_n$ is the complete graph on $n$ vertices and $\overline{m} =|E(\overline{G})| = {n\choose2}- m.$ The degree of a vertex $u$ in $G$ is denoted by $d_G(u)$. Then $d_{\overline{G}}(u) = n-1-d_G(u)$. Let $\Delta$ and $\delta$ denote the maximum vertex degree and the minimum vertex degree of the graph $G$, respectively.
\section{Sombor coindex of graphs}
Motivated by the geometric interpretation of the degree radius of an edge $uv$, which is the distance from the origin to the ordered pair $(d_G(u),d_G(v))$, where $d_G(u) \leq d_G(v)$, I. Gutman recently introduced a new vertex-degree-based molecular structure descriptor, the Sombor index, which is defined as $$SO(G)=\sum_{uv\in E(G)} \sqrt{{d_G(u)}^2+{d_G(v)}^2}.$$
We present some examples of the Somber index. For more details we refer the reader to {\cite8}.
\begin{example} 
\begin{enumerate}[label=(\roman*)]
\item $SO(K_n)=n(n-1)^2/\sqrt{2}$ and $SO(\overline{K_n}) =0.$
\item  For the cycle $C_n$ ($n\geq3$), $SO(C_n)=2\sqrt{2}n$.
\item Notice that $SO(P_2)=SO(K_2)=\sqrt{2}$. For $n\geq3$, $SO(P_n)=2(n-3)\sqrt{2}+2\sqrt{5}.$
\end{enumerate}
\end{example}
\begin{definition} The Sombor coindex of $G$ is defined as $$\overline{SO}(G)=\sum_{uv\notin E(G)} \sqrt{{d_G(u)}^2+{d_G(v)}^2}.$$
\end{definition}
\begin{remark} The Sombor coindex of $G$ is also expressed as $$\overline{SO}(G)=\sum_{uv\in E(\overline{G})} \sqrt{{d_G(u)}^2+{d_G(v)}^2}.$$
Notice that in the definition of Sombor coindex of $G$, the defining sums run over $E(\overline{G})$ but the degrees of the vertices are with respect to $G$. It is, therefore, not equal to the Sombor index of $\overline{G}$.
\end{remark}
\begin{example}
\begin{enumerate}[label=(\roman*)]
\item In the case of complete graphs, the defining sums in the Sombor coindex are taken over the empty set of edges and in the case of empty graphs, all degrees are zero.
$$\overline{SO}(K_n)=\overline{SO}(\overline{K_n}) =0 .$$
\item Notice that the complement of the cycle $C_n$ ($n\geq3$) has $n(n-3)/2$ edges and since it is a 2-regular graph, the Sombor coindex of the cycle $C_n$ is given by $\overline{SO}(C_n)=n(n-3)\sqrt{2}.$
\item By an $(x,y)$-edge of a complement graph $\overline{G}$, we mean an edge $e=uv$ where $d_G(u)=x$ and $d_G(v)=y$. Let $n\geq3$. The complement of the path $P_n$ has only one $(1,1)$-edge, $2(n-3)$ number of $(1,2)$-edges and $n-4 \choose 2$ number of $(2,2)$-edges. Thus, $$\overline{SO}(P_n)=[(n-4)(n-3)+1]\sqrt{2}+2(n-3)\sqrt{5}.$$ 
\end{enumerate}
\end{example}
\section{Properties of Sombor coindex of graphs}
In this section, we give some basic properties of $\overline{SO}(G)$ and we also give several bounds on $\overline{SO}(G)$ in terms of some useful graph parameters. From the definition of Sombor coindex, the following observations can be made immediately.
\begin{remark}
For a graph $G$, we have
\begin{enumerate}[label=(\roman*)]
\item $\overline{SO}(G-e)>\overline{SO}(G)$, where $e$ is any edge in $G$,
\item $\overline{SO}(G+e)<\overline{SO}(G)$, where the edge $e=uv$ and $u$ is not adjacent to $v$ in $G$.
\end{enumerate}
\end{remark}

First we give the upper and lower bounds on $\overline{SO}(G)$ in terms of $n$, $\Delta$ and $\delta$.
\begin{theorem} Let $G$ be a graph on $n$ vertices and $m$ edges. Then
$$\dfrac{\delta n}{\sqrt{2}}\left(n-1-\Delta\right)\leq \overline{SO}(G) \leq \dfrac{\Delta n}{\sqrt{2}}\left(n-1-\delta\right).$$
Equality holds if $G$ is a regular graph.
\end{theorem}
\begin{proof} We just prove the upper bound as the lower bound can be similarly obtained. From the definition of the Sombor coindex, we have
$$\overline{SO}(G)=\sum_{uv\notin E(G)} \sqrt{{d_G(u)}^2+{d_G(v)}^2}=\sum_{uv\in E(\overline{G})} \sqrt{{d_G(u)}^2+{d_G(v)}^2}
\leq \sqrt{2}\Delta \overline{m},$$ where the equality holds if $d_G(u)=\Delta$ for each $u\in V(G)$ i.e., $G$ is regular. 
Notice that $2\overline{m}=n(n-1)-2m$. And by the Handshaking lemma, we have 
$$2m=\sum_{u\in V(G)} d_G(u)\geq n\delta,$$ where the equality holds if $G$ is regular. 
Thus, $2\overline{m}\leq n(n-1-\delta)$. It follows that $\overline{SO}(G)\leq \dfrac{\Delta n}{\sqrt{2}}\left(n-1-\delta\right),$ where the equality holds if $G$ is regular. 
\end{proof}
As corollary we can compute the Somber coindex of $r$-regular graphs.
\begin{corollary} Let $G$ be an $r$-regular graph on $n$ vertices. Then
 $$\overline{SO}(G)=\dfrac{nr(n-1-r)}{\sqrt{2}}.$$
\end{corollary}
\begin{remark} Let $G$ be a graph on $n$ vertices. Since $\Delta\leq n-1$, we have
$$\overline{SO}(G)\leq \dfrac{n(n-1)(n-1-\delta)}{\sqrt{2}}.$$
\end{remark}
Next, we have another upper bound for the Sombor coindex in terms of $\overline{m}$, $\delta$ and $\overline{M}_1(G)$.
\begin{theorem} \label{theorem1}
Let $G$ be a graph with $m$ edges. Then $$\overline{SO}(G) \leq \overline{M}_1(G)-(2-\sqrt{2})\delta \overline{m}.$$
Equality holds if $G$ is a regular graph.
\end{theorem}
\begin{proof}
By the definition of the Sombor coindex, we have
$$\overline{SO}(G)=\sum_{uv\notin E(G)} \sqrt{{d_G(u)}^2+{d_G(v)}^2}=\sum_{uv\in E(\overline{G})} \sqrt{{d_G(u)}^2+{d_G(v)}^2}.$$ Notice that for any $u$ and $v$ with $d_G(u)\geq d_G(v)$, we have $$\sqrt{{d_G(u)}^2+{d_G(v)}^2}\leq d_G(u)+(\sqrt{2}-1)d_G(v)$$ where the equality holds if $d_G(u)=d_G(v)$. Thus, 
\begin{align*}
\overline{SO}(G)\leq &\sum_{\underset{d_G(u)\geq d_G(v)} {uv\in E(\overline{G})}}[d_G(u)+(\sqrt{2}-1)d_G(v)]\\
=&\sum_{uv\in E(\overline{G})}[d_G(u)+d_G(v)]-\sum_{\underset{d_G(u)\geq d_G(v)} {uv\in E(\overline{G})}}(2-\sqrt{2})d_G(v)\\
\leq&  \  \overline{M}_1(G)-(2-\sqrt{2})\delta \overline{m}.
\end{align*}
Further, the equality holds if $G$ is a regular graph.
\end{proof}
In \cite{1} it is proven that $\overline{M}_1(G)=2m(n-1)- M_1(G)$ and thus, we have the following corollary.
\begin{corollary} Let $G$ be a graph on $n$ vertices and $m$ edges. Then
 $$\overline{SO}(G) \leq 2m(n-1) -M_1(G)-\left(1-\dfrac{1}{\sqrt{2}}\right)[n(n-1)-2m]\delta.$$
 Equality holds if $G$ is a regular graph.
\end{corollary}
We have remarked that the Sombor coindex of a graph $G$ is not the same as the Sombor index of $\overline{G}$. However, they are related closely as follows.
\begin{theorem}  Let $G$ be a graph on $n$ vertices and $m$ edges. Then
\begin{enumerate}[label=(\roman*)]
\item $SO(G)+\overline{SO}(G)\leq \dfrac{n(n-1)\Delta}{\sqrt{2}}.$ Equality holds if $G$ is a regular graph.
\item $SO(\overline{G})+\overline{SO}(G)\leq \overline{m}(n-1+\Delta-\delta)\sqrt{2}.$ Equality holds if $G$ is a regular graph.
\end{enumerate}
\end{theorem}
\begin{proof}
By the definition of the Sombor index and the Sombor coindex, we have
\begin{equation}
\label{eqn:eq1}
SO(G)=\sum_{uv\in E(G)} \sqrt{{d_G(u)}^2+{d_G(v)}^2}\leq m\sqrt{2}\Delta
\end{equation}
\begin{equation}
\label{eqn:eq2}
\overline{SO}(G)=\sum_{uv\in E(\overline{G})} \sqrt{{d_G(u)}^2+{d_G(v)}^2}\leq \overline{m}\sqrt{2}\Delta
\end{equation}
From (\ref{eqn:eq1}) and (\ref{eqn:eq2}), we have
$$SO(G)+\overline{SO}(G)\leq(m+\overline{m})\sqrt{2}\Delta\leq {n\choose2}\sqrt{2}\Delta=\dfrac{n(n-1)\Delta}{\sqrt{2}}.$$
Further, it is easy to see that the equality holds if $G$ is a regular graph. This proves the first part. For the second part, notice that $d_{\overline{G}}(u) = n-1-d_G(u)$. Hence
\begin{align}
\label{eqn:eq3}
SO(\overline{G})=&\sum_{uv\in E(\overline{G})} \sqrt{{d_{\overline{G}}(u)}^2+{d_{\overline{G}}(v)}^2}\notag\\
=& \sum_{uv\in E(\overline{G})} \sqrt{(n-1-d_G(u))^2+(n-1-d_G(v))^2}\notag\\
\leq & \ \overline{m}(n-1-\delta)\sqrt{2}
\end{align}
From (\ref{eqn:eq2}) and (\ref{eqn:eq3}), we have $$SO(\overline{G})+\overline{SO}(G)\leq \overline{m}(n-1+\Delta-\delta)\sqrt{2}.$$ Moreover, the equality holds if $G$ is a regular graph. This completes the proof.
\end{proof}

\begin{theorem}
Let $G$ be a graph with $m$ edges. Then
$$\overline{SO}(G)+\overline{SO}(\overline{G})\leq 2\overline{M}_1(G)-(2-\sqrt{2})\delta{n \choose 2}.$$
 Equality holds if $G$ is a regular graph.
\end{theorem}
\begin{proof}
Applying Theorem \ref{theorem1} to $\overline{G}$, we have $$\overline{SO}(\overline{G}) \leq \overline{M}_1(\overline{G})-(2-\sqrt{2})\delta m.$$
 Equality holds if and only if $G$ is a regular graph.
Now, $$\overline{SO}(G)+\overline{SO}(\overline{G})\leq \overline{M}_1(G)-(2-\sqrt{2})\delta \overline{m}+ \overline{M}_1(\overline{G})-(2-\sqrt{2})\delta m.$$
It is proven in  \cite{1} that $\overline{M}_1(G)=\overline{M}_1(\overline{G})$. It follows that
\begin{align*}
\overline{SO}(G)+\overline{SO}(\overline{G})\leq & 2\overline{M}_1(G)-(2-\sqrt{2})(\overline{m}+ m)\delta \\
=&2\overline{M}_1(G)-(2-\sqrt{2})\delta{n \choose 2}.
\end{align*}
Moreover, the equality holds if $G$ is a regular graph. 
\end{proof}
\section{Relations between Sombor coindex and some coindices}
We now recall the following well-known inequality which is needed for our results concerning the relation of Sombor coindex with forgotten coindex and also with Zagreb coindices.
\begin{lemma}[P\'olya-Szeg\"o inequality \cite7] 
\label{lemma1}
Let $a_1,a_2,\dots,a_m$ and $b_1,b_2,\dots,b_m$ be two sequences of positive real numbers. If there exists real numbers $A,a,B$ and $b$ such that $0<a\leq a_k\leq A < \infty$ and $0<b\leq b_k\leq B < \infty$ for $k=1,2,\dots,m$ then $$ \dfrac{\displaystyle\sum_{k=1}^m a^2_k \displaystyle\sum_{k=1}^m b^2_k}{\left(\displaystyle\sum_{k=1}^m a_kb_k\right)^2}\leq \dfrac{(ab+AB)^2}{4abAB}$$ where the equality holds if and only if $$p=m\dfrac{A}{a}\bigg/{\left(\dfrac{A}{a}+\dfrac{B}{b}\right)}, \ q=m\dfrac{B}{b}\bigg/{\left(\dfrac{A}{a}+\dfrac{B}{b}\right)}$$ are integers and if $p$ of the numbers $a_1,a_2,\dots,a_m$ are equal to $a$ and $q$ of these numbers are equal to $A$, and if the corresponding numbers $b_k$ are equal to $B$ and $b$, respectively.
\end{lemma}

\begin{remark} The upper bound of Sombor coindex involving forgotten coindex is an easy consequence of Cauchy-Schwarz inequality. More precisely, for a graph with $m$ edges $\overline{SO}(G)\leq \sqrt{\overline{m} \overline{F}(G)}$, where the equality holds if $G$ is a regular graph. Here, we present a lower bound for $\overline{SO}$ which is still sharp for a regular graph.
\end{remark}
\begin{theorem} Let $G$ be a graph on $n$ vertices and $m$ edges. Then $$\sqrt{\overline{m}\overline{F}(G)}\leq \dfrac{1}{2} \left(\dfrac{\delta}{\Delta}+\dfrac{\Delta}{\delta}\right) \overline{SO}(G).$$
\end{theorem}
\begin{proof}  Let $V(G)=\{v_1,v_2,\dots,v_n\}$ and let $d_i$ denotes the degree of the vertex $v_i$ in $G$. Letting $a_k\to \sqrt{{d_i}^2+{d_j}^2}$ and $b_k=\delta$ in Lemma \ref{lemma1} and choosing $a=\delta=b$ and $A=\Delta=B$, we have $0<a\leq a_k\leq A < \infty$ and $0<b\leq b_k\leq B < \infty$ for $k=1,2,\dots,m$. Notice that $ \dfrac{(ab+AB)^2}{4abAB}= \dfrac{1}{4} \left(\dfrac{\delta}{\Delta}+\dfrac{\Delta}{\delta}\right)^2.$ Applying the Lemma \ref{lemma1} with the sums running over the edges in $\overline{G}$, we have
$$\dfrac{\displaystyle\sum_{v_iv_j\in E(\overline{G})}[{d_i}^2+{d_j}^2 ]\displaystyle\sum_{v_iv_j\in E(\overline{G})}{\delta}^2}{\left(\displaystyle\sum_{v_iv_j\in E(\overline{G})}\delta\sqrt{{d_i}^2+{d_j}^2}\right)^2}\leq \dfrac{1}{4} \left(\dfrac{\delta}{\Delta}+\dfrac{\Delta}{\delta}\right)^2 $$
 $$i.e., \dfrac{\overline{F}(G)\overline{m}}{\overline{SO}(G)^2}\leq \dfrac{1}{4} \left(\dfrac{\delta}{\Delta}+\dfrac{\Delta}{\delta}\right)^2 $$
 Thus $$\sqrt{\overline{m}\overline{F}(G)}\leq \dfrac{1}{2} \left(\dfrac{\delta}{\Delta}+\dfrac{\Delta}{\delta}\right) \overline{SO}(G).$$
\end{proof}
We now present the relation between Sombor coindex and the first Zagreb coindex.
\begin{theorem} Let $G$ be a graph on $n$ vertices and $m$ edges. Then $$2\sqrt{\overline{m} \Delta\overline{M}_1(G)}\leq  \left(1+\dfrac{\Delta}{\delta}\right) \overline{SO}(G).$$
\end{theorem}
\begin{proof} Let $V(G)=\{v_1,v_2,\dots,v_n\}$ and let $d_i$ denotes the degree of the vertex $v_i$ in $G$. Letting $a_k\to \sqrt{d_i+d_j}$ and $b_k \to \sqrt{\dfrac{{d_i}^2+{d_j}^2}{d_i+d_j}}$ in Lemma \ref{lemma1} and choosing $a=\sqrt{2\delta}$, $A=\sqrt{2\Delta}$, $b=\sqrt{\delta}$ and $B=\sqrt{\Delta}$, we have $0<a\leq a_k\leq A < \infty$ and $0<b\leq b_k\leq B < \infty$ for $k=1,2,\dots,m$. Notice that $$ \dfrac{(ab+AB)^2}{4abAB}= \dfrac{1}{4\delta\Delta} \left(\Delta+\delta \right)^2.$$ Applying the Lemma \ref{lemma1} with the sums running over the edges in $\overline{G}$, we have
$$\dfrac{\displaystyle\sum_{v_iv_j\in E(\overline{G})}[d_i+d_j]\displaystyle\sum_{v_iv_j\in E(\overline{G})}{\dfrac{{d_i}^2+{d_j}^2}{d_i+d_j}}}{\left(\displaystyle\sum_{v_iv_j\in E(\overline{G})}\sqrt{{d_i}^2+{d_j}^2}\right)^2}\leq  \dfrac{1}{4\delta\Delta} \left(\Delta+\delta \right)^2$$
Notice that $\dfrac{{d_i}^2+{d_j}^2}{d_i+d_j}\geq \delta$. So, $\displaystyle\sum_{v_iv_j\in E(\overline{G})}{\dfrac{{d_i}^2+{d_j}^2}{d_i+d_j}}\geq \overline{m} \delta$. Thus
 $$ \dfrac{\overline{M}_1(G)\overline{m}\delta}{\overline{SO}(G)^2}\leq \dfrac{1}{4\delta\Delta} \left(\Delta+\delta \right)^2$$
 Hence $$2\sqrt{\overline{m} \Delta \overline{M}_1(G)}\leq \left(1+\dfrac{\Delta}{\delta}\right) \overline{SO}(G).$$
\end{proof}
Lastly, we present the relation of Sombor coindex and the second Zagreb coindex.
\begin{theorem}
Let $G$ be a graph on $n$ vertices and $m$ edges. Then $$ \overline{SO}(G)\leq \sqrt{\left(\dfrac{\delta}{\Delta}+\dfrac{\Delta}{\delta}\right)\overline{m} \overline{M}_2(G)}$$ Equality holds if $G$ is a regular graph.
\end{theorem}
\begin{proof}
We first recall the Cauchy-Schwarz inequality. Let $a_1,a_2,\dots,a_m$ and $b_1,b_2,\dots,b_m$ be two sequences of real numbers. Then $${\left(\displaystyle\sum_{k=1}^m a_kb_k\right)^2} \leq \displaystyle\sum_{k=1}^m a^2_k \displaystyle\sum_{k=1}^m b^2_k.$$
Let $V(G)=\{v_1,v_2,\dots,v_n\}$ and let $d_i$ denotes the degree of the vertex $v_i$ in $G$. Letting $a_k\to \sqrt{d_id_j}$ and $b_k \to \sqrt{\dfrac{{d_i}^2+{d_j}^2}{d_id_j}}$ in the Cauchy-Schwarz inequality with the sums running over the edges in $\overline{G}$, we have
\begin{align}\label{eqn:eq4}
{\left(\displaystyle\sum_{v_iv_j\in E(\overline{G})}\sqrt{{d_i}^2+{d_j}^2}\right)^2}\leq \displaystyle\sum_{v_iv_j\in E(\overline{G})}d_id_j \displaystyle\sum_{v_iv_j\in E(\overline{G})}\dfrac{{d_i}^2+{d_j}^2}{d_i+d_j}
\end{align} 
Since $0<\delta \leq d_i\leq \Delta$ for any $v_i$, we have $\dfrac{\delta}{\Delta}\leq \dfrac{d_i}{d_j}\leq \dfrac{\Delta}{\delta}$. Now for any edge $v_iv_j$ of $G$ $(d_i\geq d_j)$, we have
\begin{align*}
\left(\dfrac{d_i}{d_j}+\dfrac{d_j}{d_i}\right)^2 = &\left(\dfrac{d_i}{d_j}-\dfrac{d_j}{d_i}\right)^2 +4 \\
\leq & \left(\dfrac{\Delta}{\delta}-\dfrac{\delta}{\Delta}\right)^2 +4 = \left(\dfrac{\Delta}{\delta}+\dfrac{\delta}{\Delta}\right)^2
\end{align*}
Thus (\ref{eqn:eq4}) becomes
$$ \overline{SO}(G)\leq \sqrt{\left(\dfrac{\delta}{\Delta}+\dfrac{\Delta}{\delta}\right)\overline{m} \overline{M}_2(G)}$$ Moreover, it is easy to see that the equality holds if $G$ is a regular graph.
\end{proof}
\section{Bounds on the Sombor coindex of graph operations}
Since several complicated and important graphs often arise from simpler graphs via some graph operations, we also present the Sombor coindex of some graph operations in this section. We give the bounds on the Sombor coindex of some graph operations namely, union, sum, compostion and Cartesian product. As an application, the Sombor coindex of some well-known (chemical) graphs are computed.
\subsection{Union}
We now consider the simplest graph operation of two graphs. A union $G_1 \cup G_2$ of two graphs $G_1$ and $G_2$ with disjoint vertex sets $V(G_1)$ and $V(G_2)$ is the graph with the vertex set $V(G_1)\cup V(G_2)$ and the edge set $E(G_1)\cup E(G_2).$ 
\begin{theorem} Let $G_1$ and $G_2$ be two graphs on $n_1$ and $n_2$ vertices, respectively. Then we have the following.
\begin{enumerate}[label=(\roman*)]
\item $ \overline{SO}(G_1\cup G_2)\leq \overline{SO}(G_1) + \overline{SO}(G_2)+n_1n_2\sqrt{{\Delta_1}^2+{\Delta_2}^2}.$
\item $\overline{SO}(G_1\cup G_2)\geq \overline{SO}(G_1) + \overline{SO}(G_2)+n_1n_2\sqrt{{\delta_1}^2+{\delta_2}^2}$.
\end{enumerate}
Here, $\Delta_i$ and $\delta_i$ denote the maximum degree vertex and the minimum degree vertex of $G_i$, respectively for $i=1,2$. Moreover, the equality holds if $G_1$ and $G_2$ are regular. 
\end{theorem}
\begin{proof} Let $G=G_1\cup G_2$. By the definition of Sombor coindex, we have
\begin{align*}
\overline{SO}(G)=&\sum_{uv\in E(\overline{G})} \sqrt{{d_G(u)}^2+{d_G(v)}^2} \\
=&\sum_{uv\in E(\overline{G}_1)} \sqrt{{d_{G_1}(u)}^2+{d_{G_1}(v)}^2}+\sum_{uv\in E(\overline{ G}_2)} \sqrt{{d_{G_2}(u)}^2+{d_{G_2}(v)}^2}\\
+&\sum_{u\in V(G_1)}\left[\sum_{v\in V(G_2)} \sqrt{{d_{G_1}(u)}^2+{d_{G_2}(v)}^2}\right] \\
\end{align*}
Notice that the last sum is the contribution to the Sombor coindex of the union from the missing edges between the components, which are the edges of the complete bipartite graph $K_{n_1,n_2}$. Thus,
$$\overline{SO}(G_1\cup G_2)\leq \overline{SO}(G_1) + \overline{SO}(G_2)+n_1n_2\sqrt{{\Delta_1}^2+{\Delta_2}^2}.$$
Moreover, it is easy to notice that the equality holds if $G_1$ and $G_2$ are regular. Similarly, the lower bound follows.
\end{proof}
\subsection{Sum}
Next we consider the sum, also known as join, of two graphs. A sum $G_1 + G_2$ of two graphs $G_1$ and $G_2$ with disjoint vertex sets $V(G_1)$ and $V(G_2)$ is the graph with the vertex set $V(G_1)\cup V(G_2)$  and the edge set $E(G_1)\cup E(G_2)\cup \{u_1u_2:u_1\in V(G_1), u_2\in  V(G_2)\}$. Hence, we keep all edges of both the graphs and also join each vertex of one graph to each vertex of the other graph. 
\begin{theorem} Let $G_1$ and $G_2$ be two graphs on $n_1$ and $n_2$ vertices and $m_1$ and $m_2$ edges, respectively. Then we have the following.
\begin{enumerate}[label=(\roman*)]
\item $ \overline{SO}(G_1 + G_2)\leq \sqrt{2}[\overline{m}_1(\Delta_1+n_2) + \overline{m}_2(\Delta_2+n_1)].$
\item $\overline{SO}(G_1+ G_2)\geq \sqrt{2}[\overline{m}_1(\delta_1+n_2) + \overline{m}_2(\delta_2+n_1)].$
\end{enumerate}
Here, $\Delta_i$ and $\delta_i$ denote the maximum degree vertex and the minimum degree vertex of $G_i$, respectively for $i=1,2$. Moreover, the equality holds if $G_1$ and $G_2$ are regular. 
\end{theorem}
\begin{proof} Let $G=G_1+G_2$. Notice that $d_G (u) = d_{G_1} (u)+n_2$ and $d_G(v) = d_{G_2} (v) + n_1$ for $u\in  V(G_1), v \in V(G_2)$. Since all possible edges between $G_1$ and $G_2$ are present in $G$, there are no missing edges, and hence their contribution is zero. Thus,
\begin{align*}
\overline{SO}(G)=&\sum_{uv\in E(\overline{G}_1)} \sqrt{{d_{G} (u)}^2+{d_{G} (v)}^2}+\sum_{uv\in E(\overline{ G}_2)} \sqrt{{d_{G} (u)}^2+{d_{G} (v)}^2}\\
=&\sum_{uv\in E(\overline{G}_1)} \sqrt{(d_{G_1} (u)+n_2)^2+(d_{G_1} (v)+n_2)^2}\\+& \sum_{uv\in E(\overline{ G}_2)} \sqrt{(d_{G_2} (u) + n_1)^2+(d_{G_2} (v) + n_1)^2}\\
\leq&\sqrt{2}[\overline{m}_1(\Delta_1+n_2) + \overline{m}_2(\Delta_2+n_1)]
\end{align*}
Moreover, the equality holds if $G_1$ and $G_2$ are regular. Similarly, the lower bound follows.
\end{proof}
\begin{corollary} 
\label{corollary1}
The Sombor coindex of the complete bipartite graph $K_{p,q}$ is given by
 $$ \overline{SO}(K_{p,q}) =\overline{SO}(\overline{K_p}+\overline{K_q})= \dfrac{pq(p+q-2)}{\sqrt{2}}.$$
\end{corollary}
\begin{remark}
We thus obtain explicit formulae for the Sombor coindex of the $n$-vertex star graph $S_n = K_{1,n-1} $ for
$ n \geq  2$ via Corollary \ref{corollary1}, i.e., $$ \overline{SO}(S_n)= \dfrac{(n-1)(n-2)}{\sqrt{2}}.$$
\end{remark}
\subsection{Cartesian product} 
The Cartesian product $G_1\square G_2$ of graphs $G_1$ and $G_2$ is the graph with the vertex set $V(G_1) \times V(G_2)$ in which $u = (u_1, u_2)$ is adjacent with $v = (v_1, v_2)$ whenever ($u_1 = v_1$ and $u_2 v_2\in E(G_2)$) or ($u_2 = v_2$ and $u_1 v_1\in  E(G_1)$). Notice that the number of edges in $G_1\square G_2$ is $n_1m_2 + m_1n_2$ and the degree of a vertex $(u_1, u_2)$ of $G_1\square G_2$ is $d_{G_1} (u_1) + d_{G_2} (u_2)$, where $n_i=|V(G_i)|$ and $m_i=|E(G_i)|$ for $i=1,2$.
\begin{theorem}
\label{theorem2}
 Let $G_1$ and $G_2$ be two graphs on $n_1$ and $n_2$ vertices and $m_1$ and $m_2$ edges, respectively. Then 
$$\overline{m} \sqrt{2}(\delta_1+\delta_2)\leq \overline{SO}(G_1 \square G_2)\leq \overline{m} \sqrt{2}(\Delta_1+\Delta_2).$$
Here, $\Delta_i$ and $\delta_i$ denote the maximum degree vertex and the minimum degree vertex of $G_i$, respectively for $i=1,2$; and $\overline{m}$ is the number of edges in $\overline{G_1 \square G_2}$. Moreover, the equality holds if $G_1$ and $G_2$ are regular. 
\end{theorem}
\begin{proof} Let $G=G_1\square G_2$. Let $n=|V(G)|$ and $m=|E(G)|$. Notice that $n=n_1n_2$ and $m=n_1m_2 + m_1n_2$. So, the number of edges in $\overline{G}$, $\overline{m}={n_1n_2\choose2}-n_1m_2 - m_1n_2$.
By the definition of the Sombor coindex, we have
\begin{align*}
\overline{SO}(G)=&\sum_{uv\in E(\overline{G})} \sqrt{{d_{G} (u)}^2+{d_{G} (v)}^2}\\
=&\sum_{uv\in E(\overline{G})} \sqrt{(d_{G_1} (u_1) + d_{G_2} (u_2))^2+(d_{G_1} (v_1) + d_{G_2} (v_2))^2}\\
\leq&\sum_{uv\in E(\overline{G})} \sqrt{2}(\Delta_1+\Delta_2)=\overline{m} \sqrt{2}(\Delta_1+\Delta_2)
\end{align*}
Moreover, the equality holds if $G_1$ and $G_2$ are regular. Similarly, the lower bound follows.
\end{proof}
The following corollary is immediate for the graph $G=C_p \square C_q$. This graph is called $C_4$ nanotorus.
\begin{corollary} The Sombor coindex of the $C_4$ nanotorus is given by
 $$ \overline{SO}(C_p \square C_q) = 2pq(pq-5)\sqrt{2}.$$
\end{corollary}
\subsection{Composition}
The composition $G_1[G_2]$ of graphs $G_1$ and $G_2$ with disjoint vertex sets and edge sets is the graph with the vertex set $V(G_1) \times V(G_2)$ in which $u = (u_1, u_2)$ is adjacent with $v = (v_1, v_2)$ whenever ($u_1$ is adjacent with $v_1$)  or ($u_1 = v_1$ and $u_2$ is adjacent with $v_2$). Notice that the number of edges in $G_1[G_2]$ is $n_1m_2 + m_1{n}^2_2$ and the degree of a vertex $(u_1, u_2)$ of $G_1[G_2]$ is $n_2  d_{G_1} (u_1) + d_{G_2} (u_2)$, where $n_i=|V_i|$ and $m_i=|E_i|$ for $i=1,2$.
\begin{theorem} Let $G_1$ and $G_2$ be two graphs on $n_1$ and $n_2$ vertices, respectively. Then we have the following.
$$\overline{m} \sqrt{2}(n_2\delta_1+\delta_2)\leq \overline{SO}(G_1[G_2])\leq \overline{m} \sqrt{2}(n_2\Delta_1+\Delta_2).$$
Here, $\Delta_i$ and $\delta_i$ denote the maximum degree vertex and the minimum degree vertex of $G_i$, respectively for $i=1,2$; and $\overline{m}$ is the number of edges in $\overline{G_1 [G_2]}$. Moreover, the equality holds if $G_1$ and $G_2$ are regular. 
\end{theorem}
\begin{proof}
Let $G=G_1[G_2]$. Let $n=|V(G)|$ and $m=|E(G)|$. Notice that $n=n_1n_2$ and $m=n_1m_2 + m_1n^2_2$. So, the number of edges in $\overline{G}$, $\overline{m}={n_1n_2\choose2}-n_1m_2 - m_1n^2_2$.
By the definition of the Sombor coindex, we have
\begin{align*}
\overline{SO}(G)=&\sum_{uv\in E(\overline{G})} \sqrt{{d_{G} (u)}^2+{d_{G} (v)}^2}\\
=&\sum_{uv\in E(\overline{G})} \sqrt{(n_2d_{G_1} (u_1) + d_{G_2} (u_2))^2+(n_2d_{G_1} (v_1) + d_{G_2} (v_2))^2}\\
\leq&\sum_{uv\in E(\overline{G})} \sqrt{2}(n_2\Delta_1+\Delta_2)=\overline{m} \sqrt{2}(n_2\Delta_1+\Delta_2)
\end{align*}
Moreover, the equality holds if $G_1$ and $G_2$ are regular. Similarly, the lower bound follows. This completes the proof.
\end{proof}
As a corollary, the Sombor coindex of the closed fences $C_n[K_2]$ is immediate.
\begin{corollary} The Sombor coindex of the closed fences $C_n[K_2]$ is given by
 $$ \overline{SO}(C_n[K_2]) = 5[2n(n-3)+4]\sqrt{2}.$$
\end{corollary}

\section{Conclusion}\label{sec13}
 In the mathematical and chemical literature, several dozens of vertex-degree-based graph invariants (usually referred to as topological indices) have been introduced and extensively studied. We define a new topological index of a graph in this paper, we call it Sombor coindex. We give several properties of the Somber coindex and its relations to the Sombor index, the Zagreb (co)indices, forgotten coindex and other important graph parameters. We also compute the bounds of the Somber coindex of some graph operations and compute the Sombor coindex for some chemical graphs as an application. One could explore further relations between Sombor coindex and other well-known (co)indices. One could also explore Sombor coindex of other graph operations which are of considerable chemical interest, such as splices and links of two or more graphs. Finding the chemical applications of this Sombor index is an attractive task for the near future.
\section*{Acknowledgments}
The second author is supported by the MATRICS project funded by DST-SERB, government of India under Grant no MTR/2017/000403 dated 06/06/2018.

\end{document}